\documentclass[10pt]{article}
\usepackage{amsmath,amssymb,amsthm}

\title{Nonabelian Rauzy-Veech Renormalization}
\author{Dmitri Scheglov}

\theoremstyle{plain}
\newtheorem{Lemma}{Lemma}[section]
\newtheorem{Theorem}{Theorem}[section]

\newtheorem{defn}{Definition}   
\begin{document}

\maketitle
\setlength{\parindent}{0pt}

\begin{abstract}
\noindent

For any Lie group $G$ we introduce a renormalization map $\mathcal{R}$ on the space of simple $G$-extensions of interval exchange transformations. Using  $\mathcal{R}$ we prove weak mixing and cohomological non-equivalence for typical simple compact $G$-extensions of IETs. This extends a well-known result of Avila and Forni for $G=U(1)$ to any compact connected Lie group $G$. 
\

It is also the first result on nonabelian extensions of interval exchange transformations.
\end{abstract}

\section{Introduction}
\subsection{ \textit{The results and aim of the paper}}

The aim of this paper is two-fold. First we introduce a nonabelian version $\mathcal{R}$ of Rauzy-Veech induction (or Rauzy-Veech renormalization) with suitable modifications known also as Rauzy-Veech-Zorich induction.
\

\

 Second we apply $\mathcal{R}$ to prove ergodic properties of $G$-extensions of Interval Exchange Transformations (from now and further IETs), where $G$ is any compact connected Lie group. More precisely our main application of $\mathcal{R}$ is the following theorem:

\begin{Theorem}[Weak mixing] For a typical IET $T=(\lambda, \pi)$, not isomorphic to an irrational rotation, and a typical simple function $\phi:[0, 1]\rightarrow G$, the $G$-extension $T_{\phi}:[0, 1]\times G\rightarrow [0, 1]\times G$, given by formula $T_{\phi}(x,y)=(Tx, \phi(x)y)$, is weakly mixing. If $T=(\lambda, \pi)$ is a rotation, then $T_{\phi}$ is typically ergodic.
\end{Theorem}

Here $\lambda=(\lambda_1,\dots, \lambda_n)$, $\lambda_1+\dots + \lambda_n=1$ is a vector of interval lengths and $\phi:[0, 1]\rightarrow G$ is a function, constant on each of the intervals. The word \textit{typical} means normalized Lebesgue measure for $\lambda$ and normalized Haar measure for $\phi$.
\

\

Theorem 1.1 extends to any compact connected Lie group $G$ the result of Avila and Forni $[1]$ which crucial part, the cohomological equation, can be reinterpreted (via Anzai criterion) as a weak-mixing of a typical $U(1)$-extension over IET. The result of Avila and Forni not only gives a full measure of weakly-mixing $U(1)$-extensions, but also estimates the \textit{Hausdorff dimension} of exceptions. In case of a more general Lie group $G$ it is not (yet?) possible to obtain such a delicate information due to the presence of \textit{higher-dimensional irreducible representations}, which are absent in case of $G=U(1)$. 
\

\

As a consequence we obtain an almost sure convergence result for a stochastic process on the Lie group $G$, generated by a typical IET:

\begin{Theorem}[\textit{Convergence to Haar measure}] For a typical IET $T=(\lambda, \pi)$ of $n\geq 2$ intervals and a typical $n$-tuple $A=(A_1,\dots, A_n)\in G^n$, for almost all $x\in[0, 1]$ the sequence of discrete measures ${\mu}^k(A,x)$ converges to the normalized Haar measure $\nu$ on $G$.
\end{Theorem}
 
What Theorem 1.2 says is that for a typical IET $T=(\lambda, \pi)$ and a typical $x\in[0, 1]$ if one assigns symbols $1, \dots, n $ to the exchanged intervals  then one has an infinite coding word $ w(x)=w_1w_2 \dots w_k\dots$ in the alphabet $\{ 1,\dots,n \}$. Now an $n$-tuple $A$ produces a sequence ${\mu}^k(A,x)$ of finitely supported normalized measures on $G$, given by formula ${\mu}^k(A,x)=\frac{1}{k}({\delta}_{A_{w_1}}+{\delta}_{A_{w_2}\cdot A_{w_1}}+\dots+{\delta}_{A_{w_k}\cdots A_{w_1}})$. Theorem 1.2 asserts that ${\mu}^k(A,x)\rightarrow\nu$ for a typical pair $(A, x)$.
\

\

One more application of $\mathcal{R}$ cohomologically distinguishes different simple $G$-extensions over an IET $T$:

\begin{Theorem}$\{$Typical cohomological non-equivalence$\}$. Let $G$ be a compact connected Lie group. Then for a typical IET $T=(\lambda, \pi)$ and a typical pair of simple functions $\phi, \psi:[0, 1]\rightarrow G$, the $G$-extensions $T_{\phi}$ and $T_{\psi}$ are not measurably cohomologous.
\end{Theorem}

The extensions $T_{\phi}$ and $T_{\psi}$ are said to be \textit{measurably cohomologous} if there exists a measurable function $f:X\rightarrow G$ such that the following cohomological equation holds almost everywhere on $X$: $f(Tx)\cdot\phi(x)=\psi(x)\cdot f(x)$.  Existence of a measurable cohomology implies that $T_{\phi}$ and $T_{\psi}$ are isomorphic as measure-preserving maps. More precisely the  map 
$F: X\times G\rightarrow X\times G$, given by equation $F(x,g)=(x, f(x)g)$ satisfies $F\circ T_{\phi}= T_{\psi}\circ F$. 
\

\

Theorems 1.1, 1.2, 1.3 are first results on the ergodic theory of $G$-extensions of IETs for a nonabelian group $G$. The previous results only treat the abelian cases.
\

\

After the first draft of this paper was written, it was immediately suggested by Forni that $\mathcal{R}$ should be intimately related to Forni-Goldman's construction $[5]$ of Teichmuller flow on the variety of representations of the surface group, and so $\mathcal{R}$ is interesting by itself. As following study has shown the renormalization map $\mathcal{R}$ and Forni-Goldman's Teichmuller flow on character variety of a \textit{punctured} surface of genus $g$ are related analogously to how usual Rauzy-Veech induction is related to classical Teichmuller geodesic flow via Veech Zippered Rectangle construction. 
\

As this geometric continuous interpretation of $\mathcal{R}$ is not used in this paper we refer an interested reader to $[6]$ and $[7]$ for more details. However one should notice that, except the $U(1)$ case, Forni-Goldman Teichmuller flow \textit{can not} be applied directly to $G$-extensions of IETs. The main reason is that unlike IETs, the $G$-extensions of IETs are \textit{not}  first return maps of some naturally defined flows on $G$-bundles over flat surfaces.  So it is absolutely necessary to introduce $\mathcal{R}$ and study its ergodic properties.
\

\subsection{ \textit{Overview of the previous results}}
The Rauzy-Veech induction was originally introduced by Rauzy in $[10]$, and later extensively studied by Veech in a series of papers$[15]$,$[16]$,$[17]$. As it turned out the original Rauzy-Veech induction possessed an infinite absolutely continuous invariant measure $\mu$ and so it was later modified by Zorich $[20]$ to Rauzy-Veech-Zorich induction with a finite absolutely invariant measure $\mu$. 
\

\

The applications of Rauzy-Veech and Rauzy-Veech-Zorich induction to the study of IETs and more generally flows on compact surfaces are too numerous to be all stated here, but the one of particular interest to us is the result by Avila and Forni $[1]$, establishing weak-mixing for typical IETs and translation flows. 
\

\

The crucial part of $[1]$ deals with a particular cohomological equation and establishes an upper bound on the Hausdorff dimension of exceptional parameters. Similar cohomological equation appears when we deal with one-dimensional representations of $G$. We notice however that for example in case of semi-simple $G$ there are no one-dimensional representations, and so the \textbf{essential new ingredient} of our approach is dealing with \textit{higher-dimension representations of $G$}, which is the issue not appearing in $[1]$ as all the irreducible representations of $U[1]$ are one-dimensional.
\

\

  Given a map $T: X\rightarrow X$, which preserves a probability measure $\mu$  and a family of maps $S_x:Y\rightarrow Y$ , each preserving a probability measure $\nu$ on the measurable space $Y$, one has a skew product transformation $T\rtimes S_x: X\times Y\rightarrow X\times Y$ defined by formula $T\rtimes S_x(x,y)=( T(x), S_x(y))$ which, if measurable, preserves a measure $\mu\times\nu$.
\

\

If $G$ is a compact topological group with the Haar measure $\nu$ then one can take a measurable function $\phi: X\rightarrow G$ and form a skew product $T_{\phi}(x,y)=(Tx,\phi(x)y)$  which in this special case is called \textit{$G$-extension} of $T$ (often, \textit{skew-shift} over $T$).  
\

\

For a comprehensive survey of ergodic theory of general $G$-extensions we refer an interested reader to Parry and Pollicott$[9]$. Some more references can also be found in Lind$[7]$. Regarding more specific case, when the base map $T$ is an interval exchange transformation, one has to separate the case when $T$ is an irrational rotation, and when $T$ is an IET of $n\geq3$ intervals.
\

\

For results about skew products over irrational rotations we refer to the works of Pask$[10]$, Conze, Piekniewska$[4]$. See also Conze and Fraczek $[3]$ for more comprehensive list of references about this type of skew products. 
\

\

Regarding the results for skew-product over IETs, except already mentioned result $[1]$, Conze and Fraczek$[3]$ studied ergodic properties of cocycles with values in some locally compact abelian groups. Fraczek and Ulcigrai $[5]$ proved some non-ergodicity results for specific $\mathbb{Z}$-valued cocycles arising in the study of billiards with infinite periodic obstacles. Recently Chaika and Robertson$[2]$ have shown ergodicity of piecewise constant cocycles with values in $\mathbb{R}$ for some special class of interval exchange transformations, which they call linearly recurrent. And the most recent work of Forni  $[5]$ establishes \textit{effective} weak mixing for typical $S^1$-extensions of flows on flat surfaces, which can be thought of as an effective continuous counterpart of our results, when $G=U(1)$ as the suitable first return map of such a flow is a $U(1)$- extension of an IET.

\subsection{Checklist of notations.}

Throughout the paper we will \textit{persistently} use the following notations for the convenience of a reader:
\

\

$n$, the number of interchanged intervals
\

$\lambda=({\lambda}_1,\dots, {\lambda}_n)$, the vector of interval lengths
\

${\mathbb{R}}^n_+=\{(x_1,\dots, x_n)\in{\mathbb{R}}^n| x_i>0, 1\leq i\leq n\}$, the set of positive vectors
\

$S_n$, the symmetric group on $n$ symbols
\

$S^0_n$, the subset of $S_n$ consisting of irreducible permutations
\

$I=(0, |\lambda|)$, the interval
\

$I_k$, $1\leq k\leq n$, the interchanged subintervals of $I$
\

$T:I\rightarrow I$, an interval exchange transformation or
\

$T:X\rightarrow X$, an automorphism of the probability space $(X, \mu)$ or 
\

$T: H\rightarrow H$, a bounded linear operator on the Hilbert space $H$
\

${\Delta}_{n-1}\subseteq {\mathbb{R}}^n$, the $(n-1)$-dimensional simplex of normalized interval exchange transformations
\

$G$, compact connected Lie group
\

$\nu$, normalized Haar measure on $G$ or on $G^n$
\

$\phi$, $\psi:I\rightarrow G$, simple functions
\

$T_{\phi}$ or $T_{\psi}: I\times G\rightarrow I\times G$, the simple $G$-extensions over $T$
\

$\Upsilon$, Rauzy-Veech Renormalization map
\

$I^m$, the $m$-th $\Upsilon$-induced interval 

$\mathcal{R}$, the Extended Rauzy-Veech Renormalization map
\

$\Omega$, full measure set of IETs, satisfying Veech Properties $P_1(\epsilon, m)$ and $P_2(\epsilon, m)$
\

$S$, a compact subset of measure zero of $G^n$ or of $G^n\times G^n$ with specific properties
\

$U(d)$, the group of unitary matrices of dimension $d$
\

$\Theta: G\rightarrow U(d)$, $d$-dimensional unitary irreducible representation of $G$
\

$\gamma:G\rightarrow U(1)$, one-dimensional nontrivial unitary representation of $G$
\

$B(H)$, the space of bounded linear operators on the Hilbert space $H$

\

\textbf{Remark on the use of subscripts and superscripts.} In this paper the \textit{subscript} index means a fixed non-asymptotic index usually with the range $[1, n]$. For example $I_k$ means $k$-th exchanged interval, $g_k$ means $k$-the component of a vector $g=(g_1,\dots, g_n)\in G^n$. Whereas \textit{superscript} index means asymptotic index with unbounded range. For example $I^m$ means the $m$-th $\Upsilon$-induced interval, $T^m$ means the $m$-th pover of $T$ etc. The same rule applies when \textit{both} notations are used. For example $I^m_k$ means $k$-th subinterval of $I^m$.

\section{Acknowledgements}
We would like to thank Pablo Carrasco for several fruitful and enlightening conversations during the work on this paper.
\

\

We would like to thank Giovanni Forni for explaining to us the delicate aspects of his work $[1]$ with Avila regarding the upper bounds on the  Hausdorff dimension and the structure of the exceptional set, interest to the paper and many fruitful discussions, especially about geometric interpretation of 
$\mathcal{R}$. 
\

\

We would like to thank Andrey Gogolev for several inspiring and enlightening discussions during the work on this paper.
\

\

We would like to thank Nikolay Goussevski for many enlightening discussions and constant support during the work on this paper.
\

\

We would like to thank Narutaka Ozawa for giving us the idea on the proof of Lemma 7.2

\section{The Extended Rauzy-Veech Renormalization}

\subsection{ \textit{Interval Exchange Transformations}}

 Throughout the paper we will persistently use the following notations for the convenience of the reader. Let $n\geq 2$ and $\lambda=({\lambda}_1,...,{\lambda}_n)\in{\mathbb{R}}^n_+$ be a length vector with all positive coordinates. Let $S_n$ be a symmetric group on $n$ symbols and $\pi\in S_n$. A permutation $\pi$ is called \textit{irreducible} if for any $k, 1\leq k< n$, $\pi\{1,...,k\}\neq\{1,...,k\}$. 
  $S^0_n$ denotes the set of all irreducible permutations on $n$ symbols. The cut points are defined as $\beta_0=0$ and $\beta_k=\sum\limits_{i=1}^k\lambda_i$, $1\leq k\leq n$. Also the intervals are defined as $I_k=[{\beta}_{k-1}, {\beta}_k)$, $1\leq k\leq n$ and $I$. 
  
  \begin{defn}
  An \textbf{Interval Exchange Transformation} (from now and further \textbf{IET}) defined by pair $(\lambda, \pi)$ is a transformation 
  $T:[0, |\lambda|]\rightarrow [0,|\lambda|]$ interchanging intervals $I_k$ as solid segments, with respect to the permutation $\pi$. 
  \end{defn}
  
  Any IET is a piecewise isometry, preserving Lebesgue measure on $[0, |\lambda|]$.
 \
 
 \
 
\subsection{ \textit{Rauzy-Veech Induction}}

Given an interval exchange $T=(\lambda, \pi)$ of n intervals such that ${\lambda}_n\neq{\lambda}_{{\pi}^{-1}(n)}$ we have two possibilities:
\

\

$1)$ \textbf{Rauzy rule A.} ${\lambda}_n<{\lambda}_{{\pi}^{-1}(n)}$. In this case put $\overline{I}=[0, |\lambda|-{\lambda}_n]$
\

\

$2)$ \textbf{Rauzy rule B.} ${\lambda}_n>{\lambda}_{{\pi}^{-1}(n)}$. In this case put $\overline{I}=[0, |\lambda|-{\lambda}_{{\pi}^{-1}(n)}]$
\

\

The first return map of $T$ on $\overline{I}$ is again an IET $\overline{T}=(\overline{\lambda}, \overline{\pi})$ of $n$ intervals. The new permutation  depends only on A or B and is denoted $A\pi$ or $B\pi$.
\

\
\begin{defn} The \textbf{Rauzy-Veech induction} is the map $\Upsilon$ defined on the full measure subset of ${\mathbb{R}}^n_+\times S^0_n$ by formula $\Upsilon (\lambda, \pi)=(\overline{\lambda}, \overline{\pi})$
\end{defn}
\

\subsection{ \textit{G-extensions}}
\

Let $T:X\rightarrow X$ be a measure preserving transformation of a probability space $(X,\mu)$, $G$ be a compact topological group with the normalized Haar measure $\nu$ and $\phi: X\rightarrow G$ be a measurable function. 

\begin{defn}
The \textbf{G-extension} is a transformation $T_{\phi}: X\times G\rightarrow X\times G$ given by formula $T_{\phi}(x,y)=(Tx,\phi(x)y)$.
\end{defn}
\

 $T_{\phi}$ preserves the product measure $\mu\times\nu$.
\

\

\subsection{ \textit{Rauzy maps A and B}}
Let $G$ be a compact connected Lie group with the normalized Haar measure $\nu$. Then the Haar measure for $G^n$ is the product measure
 $\nu\times...\times\nu$ which from now and further we will also denote by $\nu$ without the risk of confusion.
 \

\begin{defn}
 The Rauzy map  \textbf{A}$:G^n\rightarrow G^n$ is defined by formula 
 \
 
 $A(g_1,...,g_n)=(h_1,...,h_n)$, where:

\begin{equation}
    h_k=
    \begin{cases}
      g_k, & \text{if}\ 1\leq k\leq {\pi}^{-1}(n) \\
      g_ng_{{\pi}^{-1}(n)}, & \text{if}\ k={\pi}^{-1}(n)+1\\
      g_{k-1}, & \text{if}\  {\pi}^{-1}(n)+2\leq k\leq n\ (such\ k\ may\ not\ exist)
    \end{cases}
  \end{equation}
  \end{defn}
  \
  
  \begin{defn}
  The Rauzy map \textbf{B}$:G_n\rightarrow G_n$ is defined by formula 
  \
  
  $B(g_1,...,g_n)=(h_1,...,h_n)$, where:
  
  \begin{equation}
    h_k=
    \begin{cases}
      g_k, & \text{if}\ 1\leq k\leq {\pi}^{-1}(n)-1 \ (such\ k\ may\ not\ exist) \\
      g_ng_{{\pi}^{-1}(n)}, & \text{if}\ k={\pi}^{-1}(n)\\
      g_{k}, & \text{if}\  {\pi}^{-1}(n)+1\leq k\leq n
    \end{cases}
  \end{equation}
  \end{defn}
  \
  
  \
  \begin{Lemma} The Rauzy maps $A$ and $B$ preserve the measure $\nu$ on $G^n$.
  \end{Lemma}
  
    \begin{proof} The maps $A$ and $B$ are compositions of \textbf{elementary Nielsen maps} 
  $N^{\alpha}_{ij}: G^n\rightarrow G^n, 1\leq i< j\leq n$ and $N^{\beta}: G^n\rightarrow G^n$ defined by 
  \
  
  $N^{\alpha}_{ij}(g_1,...,g_i,...,g_j,...,g_n)=(g_1,...,g_j,...,g_i,...,g_n)$
   and 
\
   
   $N^{\beta}(g_1, g_2,...,g_n)=(g_2g_1, g_2,...,g_n)$. Both 
  $N^{\alpha}_{ij}$ and $N^{\beta}$ are easily seen to preserve $\nu$.
  
  \end{proof}
\subsection{ \textit{Extended Rauzy-Veech induction, extended Veech Cocycle}}
Let us consider an IET $T=(\lambda, \pi)$ with permuted intervals $I_1,..., I_n$. 

\begin{defn}
The \textbf{simple function} $\phi: I\rightarrow G$ is defined as
$\phi(x)=g_k$, if $x\in I_k$, $1\leq k\leq n$, where the $n$-tuple $g=(g_1,...,g_n)\in G^n$. 
\end{defn}

\begin{defn}
The \textbf{simple G-extension}  is a G-extension  $T_{\phi}$ defined by an IET $T$ and a simple function $\phi$.
 \end{defn}

From definition 7 it follows that ${\mathbb{R}}^n_+\times S^0_n\times G^n$ is the space of simple G-extensions.
\

For a simple $G$-extension $T_{\phi}=(\lambda, \pi, g)$ let $(\overline{\lambda}, \overline{\pi})=\Upsilon(\lambda,\pi)$. One easily sees that the first return map of $T_{\phi}$ on the set $[0,|\overline{\lambda}|]\times G$ is again a simple $G$-extension given by the triple $(\overline{\lambda},\overline{\pi},\overline{g})$ where $\overline{g}=$\textbf{A}$g$ or $\overline{g}=$\textbf{B}$g$ depending on which Rauzy rule was used for $(\lambda, \pi)$. 

\begin{defn}
The \textbf{Extended Rauzy-Veech Renormalization} is a map 
\

$\mathcal{R}:{\mathbb{R}}^n_+\times S^0_n\times G^n\rightarrow {\mathbb{R}}^n_+\times S^0_n\times G^n$, defined for full measure set of $(\lambda, \pi)$ by formula
$\mathcal{R}(\lambda,\pi, g)=(\overline{\lambda},\overline{\pi},\overline{g})$.
\end{defn}

\begin{defn}
 The \textbf{Extended Veech Cocycle} is a map
 \
 
  $\Gamma:{\mathbb{R}}^n_+\times S^0_n\rightarrow Homeo({G^n})$, defined for almost every $(\lambda, \pi)$ by 
 
  \begin{equation}
    \Gamma(\lambda, \pi)g=
    \begin{cases}
      Ag, & \text{if}\ {\lambda}_n<\lambda_{{\pi}^{-1}(n)} \\
      Bg, & \text{if}\ {\lambda}_n>\lambda_{{\pi}^{-1}(n)}
    \end{cases}
  \end{equation}
\end{defn}
\

From the definitions of $\mathcal{R}$ and $\Gamma$ follows identity 
$\mathcal{R}(\lambda, \pi, g)=(\Upsilon(\lambda, \pi), \Gamma(\lambda, \pi)g)$ so $\mathcal{R}$ itself is a skew product over $\Upsilon$.
\

\section{Weak mixing for typical compact G-extensions over interval exchange transformations}
Clearly for $T_{\phi}$ to be ergodic or weakly mixing it is necessary that the base transformation $T$ itself is ergodic or weakly mixing. The sufficient condition for $T_{\phi}$ to be weakly mixing is given by the following criterion due to Keynes and Newton[6],[8],[9].
\

\begin{Theorem}\textbf{\textit{Keynes-Newton criterion of ergodicity.}}
\

\

 Let $T:X\rightarrow X$ be an ergodic measure-preserving transformation of a probability space $(X,\mu)$, $G$ be a compact Lie group with the normalized Haar measure $\nu$ and  $\phi:X\rightarrow G$ be a measurable function. Then the $G$-extension  
 $T_{\phi}:X\times G\rightarrow X\times G$ is ergodic if and only if:
\

\

$1)$ For any unitary irreducible representation $\Theta: G\rightarrow U(d)$ of dimension $d\geq 2$ the equation 

\begin{equation}
F(Tx)=\Theta(\phi(x))F(x)
\end{equation}
\

does not have nonzero solutions $F\in L^2(X, {\mathbb{C}}^d)$.
\

\

$2)$ For any non-trivial representation $\gamma: G\rightarrow U(1)$ the equation

\begin{equation}
f(Tx)=\gamma(\phi(x))f(x)
\end{equation}
\

does not have nonzero solutions $f\in L^2(X, \mathbb{C})$.
\end{Theorem}
\

\begin{Theorem}\textbf{\textit{Keynes-Newton criterion of weak mixing.}}
\

\

 Let $T:X\rightarrow X$ be a weakly mixing measure-preserving transformation of a probability space $(X,\mu)$, $G$ be a compact Lie group with the normalized Haar measure $\nu$ and  $\phi:X\rightarrow G$ be a measurable function. Then the skew shift  
 $T_{\phi}:X\times G\rightarrow X\times G$ is weakly mixing if and only if:
\

\

$1)$ For any unitary irreducible representation $\Theta: G\rightarrow U(d)$ of dimension $d\geq 2$ the equation 

\begin{equation}
F(Tx)=\Theta(\phi(x))F(x)
\end{equation}
\

does not have nonzero solutions $F\in L^2(X, {\mathbb{C}}^d)$.
\

\

$2)$ For any non-trivial representation $\gamma: G\rightarrow U(1)$ and any $\alpha\in\mathbb{C}, |\alpha|=1$ the equation

\begin{equation}
f(Tx)=\alpha\gamma(\phi(x))f(x)
\end{equation}
\

does not have nonzero solutions $f\in L^2(X, \mathbb{C})$.
\end{Theorem}
\

We now remind two properties of generic IETs by  Veech, which we will combine with Keynes-Newton criterion. Let $m\in{\mathbb{Z}}_+$ and $({\lambda}^m,{\pi}^m)={\Upsilon}^m(\lambda, \pi)$ and $I^m=[0, |{\lambda}^m|]$.

\begin{defn}$\{$Veech property $P_1(\epsilon, m)$.$\}$ An IET $T=(\lambda, \pi)$ is said to satisfy property $P_1(\epsilon, m)$ if there exists 
$b\geq \epsilon\frac{|\lambda|}{|{\lambda}^m|}$, such that ${\beta}_i(\lambda)\notin T^kI^m$, for $1\leq i\leq n-1$ and $0\leq k<b$.
\end{defn}

\begin{defn}$\{$Veech property $P_2(\epsilon, m)$.$\}$ An IET $T=(\lambda, \pi)$ is said to satisfy property $P_2(\epsilon, m)$ if 
${\lambda}_{i}^m\geq\epsilon|{\lambda}^m|$ for $1\leq i\leq n$. 
\end{defn}

\begin{Theorem}$\{$Veech$\}$. For any $n\geq 2$ there is an  $\epsilon>0$ and a full measure set $\Omega\subseteq{\mathbb{R}}^n_+\times S^0_n$ such that for any IET $T\in\Omega$ there exists an infinite set $E\in{\mathbb{Z}}_+$
\end{Theorem}

\begin{Theorem} For a full measure set $\Omega$ of IETs, $\Omega\subseteq{\mathbb{R}}^n_+\times S^0_n$, $n\geq 2$ and for all $g=(g_1, g_2,\dots, g_n)\in G^n$ the following property takes place.
\
Let $\phi(x):[0, |\lambda|]\rightarrow G$ be a simple function, constructed by $g$.   Assume that for $T\in\Omega$ and for a unitary representation $\Theta: G\rightarrow U(d)$, $d\geq 2$  the equation 

\begin{equation} 
F(Tx)=\Theta(\phi(x))F(x)
\end{equation}
\

has a nonzero solution $F\in L^2(X, {\mathbb{C}}^d)$. Denote $({\lambda}^m, {\pi}^m, g^m)= {\mathcal{R}}^m(\lambda, \pi, g)$.  Then there exists an infinite set $E(T)\subseteq{\mathbb{Z}}_+$ and a sequence of vectors $\{w^m\}\in{\mathbb{C}}^d, ||w^m||=1$, such that
 $||\Theta(g_k^m)w^m-w^m||\rightarrow 0$, $1\leq k\leq n$, for $m\in E(T)$. 
\end{Theorem}
\

\begin{proof}

Let us fix $\delta>0$ and let $\Omega$ be a full measure set of ergodic IETs satisfying the conclusion of Theorem 4.3. Pick $T\in\Omega$.  As $\Theta$ is a unitary representation then from equation 8 it follows that $||F(Tx)||=||F(x)||$. As $T$ is ergodic then without loss of generality we may assume that $||F(x)||=1, x\in[0, 1]$.
\

\

 Let $E\in{\mathbb{Z}}_+$  be an infinite set such that for each $m\in E$ , $T$ satisfies  $P_1(\epsilon, m)$ and $P_2(\epsilon, m)$.  If $m\in E$ and $I^m=[0, |{\lambda}^m|]$ then $P_1(\epsilon,m)$ implies  that $T^kI^m$ is an interval for $0\leq k<b$ ( $b$ depends on $m$) and also that $|{\cup}_{k=0}^{b-1}T^kI^m|\geq\epsilon|\lambda|$. As $m\rightarrow\infty$, $|I^m|\rightarrow 0$; therefore, if $m\in E$ is sufficiently large, then, by Lemma 7.1 there exist $k$ and $w\in{\mathbb{C}}^d, ||w||=1$ such that $0\leq k<b$ and
\

\begin{equation}
{\int}_{T^kI^m}||F(x)-w||dx<\delta|I^m|
\end{equation}
\

As $F(T^kx)=\Theta(\phi(T^{k-1}x))...\Theta(\phi(x))F(x)$ and since $k<b$ then the unitary operator $\Theta(\phi(T^{k-1}x))...\Theta(\phi(x))$  is independent on $x\in I^m$. It follows that there exists $w^m\in{\mathbb{C}}^d, ||w^m||=1$ such that

\begin{equation}
{\int}_{I^m}||F(x)-w^m||dx<\delta|I^m|
\end{equation}
\

From equation 10 it follows that the set $\{x\in I^m: ||F(x)-w^m||\geq\sqrt{\delta}\}$ has measure at most $\sqrt{\delta}|I^m|$.
\

\

Let $l_k^m$ denote the first return time of $I_k^m$ into $I^m$. By definition of $\mathcal{R}$ we have the relation $F(T^{l_k^m}x)=\Theta(g_k^m)F(x)$ for $x\in I_k^m$, for $1\leq k\leq n$. If $1\leq k\leq n$ and if there is an $x\in I_k^m$ such that $||F(x)-w^m||\leq{\sqrt{\delta}}$ and $F(T^{l_k^m}x)-w^m||\leq{\sqrt{\delta}}$ then $||\Theta(g_k^m)w^m-w^m||\leq 2\sqrt{\delta}$. The existence of such an $x$ is guaranteed by $P_2(\epsilon, m)$ if $\delta$ is small enough.

\end{proof}

\begin{Lemma} Let $\Theta: G\rightarrow U(d)$ be a $d$-dimensional unitary irreducible representation of a compact connected Lie group $G$ and $d\geq2$. Let $n\geq 2$ be a positive integer and $S\subseteq G^n$ be a set of $n$-tuples $g=(g_1,...,g_n)$ such that there exists a vector $w\in{\mathbb{C}}^d, ||w||=1$ such that $\Theta(g_k)w=w$ for $1\leq k\leq n$. Then $S$ is a compact set of measure zero with respect to the Haar measure $\nu$ on $G^n$.
\end{Lemma}

\begin{proof} The compactness of $S$ is an immediate consequence of compactness of $G$ and compactness of the unit sphere $S^{2d-1}=\{w\in {\mathbb{C}}^d | ||w||=1\}$.
We move on to prove that $S$ has zero measure. It is enough to prove that a full measure set of pairs $(g_1, g_2)\in G^2$ satisfies the property:  there \textbf{does not exist} a vector $w\in{\mathbb{C}}^d,||w||=1$ such that $\Theta(g_1)w=w$ and $\Theta(g_2)w=w$. 
\

\

It is a classical result that for any compact connected Lie group there is a set of pairs $P\in G^2$ of a full measure, such that any pair $g=(g_1, g_2)\in P$ generates a dense subgroup. For such a generating pair existence of $w\in{\mathbb{C}}^d,||w||=1$ such that $\Theta(g_1)w=w$ and $\Theta(g_2)w=w$ would imply that for any $g\in G$, $\Theta(g)w=w$ and this contradicts irreducibility of $\Theta$.

\end{proof}

\begin{Lemma} Assume $\Theta:G\rightarrow U(d)$ is a $d$-dimensional unitary irreducible representation of $G$, $d\geq 2$ and $g^m=(g_1^m,..., g_n^m)\in G^n$ is a sequence of $n$-tuples, such that there is sequence of vectors $w^m\in{\mathbb{C}}^n, ||w^m||=1$, satisfying $||\Theta(g_k^m)w^m-w^m||\rightarrow 0$, for $1\leq k\leq n$. Then $d(g^m,S)\rightarrow 0$.
\end{Lemma}
 Here $S\subseteq G^n$ is defined in Lemma 4.1 and \textbf{d} is metric on $G^n$ induced from any biinvariant Riemannian metric on $G$.
 
\begin{proof}
Assume by contradiction that $d(g^m,S)\nrightarrow 0$. By passing to subsequence we may assume that $d(g^m,S)\geq\epsilon$ for some $\epsilon>0$. As the set $S_{\epsilon}=\{g\in G^n | d(g,S)<\epsilon\} $ is clearly open, then $G^n\backslash S_{\epsilon}$ is compact. By passing to subsequence we may assume that there is an $n$-tuple $g\in G^n\backslash S_{\epsilon}$, such that $g^m\rightarrow g$.
\

\

 Moreover as $||w^m||=1$ and a unit sphere ${\textbf{S}}^{2d-1}\in{\mathbb{C}}^d$ is compact we may assume, one more time passing to subsequence, that there is a vector $w\in {\mathbb{C}}^d$, $||w||=1$, such that $w^m\rightarrow w$. Then 
$\Theta(g_k)w-w=(\Theta(g_k)w-\Theta(g_k^m)w)+(\Theta(g_k^m)w-\Theta(g_k^m)w^m)+(\Theta(g_k^m)w^m-w^m)+(w^m-w)$. Using unitarity of $\Theta$ and triangle inequality we see that the righthandside of the latter identity goes to zero which implies that $\Theta(g_k)w=w$. But this means that $g\in S$ which is not possible as $g\in G^n\backslash S_{\epsilon}$.
\end{proof} 

\begin{Lemma} Assume that $T^m: X\rightarrow X$ is a sequence of measure preserving automorphisms of a probability space $(X,\mu)$ and $A\subseteq X$ is a measurable subset. Let $Y$ be a set of points which \textbf{eventually stay in $A$}, or more formally $\forall y\in Y$ $\exists$ $m(y)\in {\mathbb{Z}}_+$ such that $\forall m\geq m(y) : T^m(y)\in A$. Then $\mu(Y)\leq\mu(A)$.

\end{Lemma}

\begin{proof}
For each non-negative integer $p$ we define the set $Y_p\in X$ as follows: 
\

$Y_p=\{y\in X|$ $ 1) \forall m\geq p: T^m(y)\in A ($ here we assume that $T^0(x)=x);$ $2)$ Either $p=0$ or $T^{p-1}(y)\notin A\}$. Informally speaking $Y_p$ is a set of points, which stay in $A$ since the time $p$ , but not since time $p-1$. Clearly the sets $Y_p$ do not intersect for $0\leq p<\infty$ and $Y=\bigcup\limits_{p=0}^{\infty}Y_p$. 
\

Now   $T^m(\bigcup\limits_{p=0}^{m}Y_p)\subseteq A$ by definition of the sets $Y_p$. As $T^m$ preserves $\mu$ we have that 
$\mu(\bigcup\limits_{p=0}^{m}Y_p)=\mu(T^m(\bigcup\limits_{p=0}^{m}Y_p))\leq\mu(A)$. As $Y=\bigcup\limits_{p=0}^{\infty}Y_p$ we have that $\mu(Y)=\lim\mu(\bigcup\limits_{p=0}^{m}Y_p)\leq\mu(A)$ Q.E.D.  
\end{proof}

\begin{Theorem}Let $d\geq 2$ and $\Theta:G\rightarrow U(d)$ be an irreducible unitary representation of $G$. Let $n\geq 2$. Then for almost all triples $(\lambda,\pi,g)\in{\Delta}_{n-1}\times S^0_n\times G^n$ the equation

\begin{equation}
F(Tx)=\Theta(\phi(x))F(x)
\end{equation}

\

has only a trivial solution $F(x)=0\in L^2([0, 1], {\mathbb{C}}^d)$
\end{Theorem}

\begin{proof} Assume that for some triple $(\lambda,\pi,g)\in{\Delta}_{n-1}\times S^0_n\times G^n$ there exists a nonzero solution $F(x)$ to the equation (11). Then by Theorem 4.4. there exists a sequence of vectors $w^m\in{\mathbb{C}}^d, ||w^m||=1$, such that
 $||\Theta(g_k^m)w^m-w^m||\rightarrow 0$, for $1\leq k\leq n$. Then Lemma 4.2 implies that $d(g_m,S)\rightarrow 0$.
 \
 
 It is enough then to prove that for any sequence ${\Gamma}^m:G^n\rightarrow G^n$ consisting of Rauzy maps $A$ and $B$, the set
 $D=\{g\in G^n| d({\Gamma}^m(g), S)\rightarrow 0\}$ has measure zero. 
 \
 
 \
 
 Choose a positive integer $p$ and consider a set 
 $S_p=\{g\in G^n| d(g, S)<1/p\}$. Then clearly the set $D$ is eventually in $S_p$ under the sequence ${\Gamma}^m$. So by Lemma 4.3 for any $p$, $\nu(D)\leq\nu(S_p)$. As set $S$ is compact it implies that $S=\bigcap S_p$. As $S_p$ is a monotone sequence of sets, $\nu(S_p)\rightarrow\nu(S)=0$ and so $\nu(D)=0$.
\end{proof}

\subsection{Adapted Avila-Forni argument for representations of dimension one.}

In order to apply Keynes-Newton criterion to one-dimensional representations of $G$ we will need the following theorem by Avila and Forni$[1]$.

\begin{Theorem}$\{$Hausdorff dimension of exceptional set$\}$
\

 For a full measure set of IETs $(\lambda, \pi)\in {\Delta}_{n-1}\times S_n^0$, $n\geq 3$ there is a set $W=W(\lambda, \pi)\subseteq{\mathbb{R}}^n$ of Hausdorff dimension at most $g(\pi)$ such that for all vectors 
$h=(h_1,...,h_n), h\in{\mathbb{R}}^n\backslash W$ the equation
\begin{equation}
F(Tx)=\phi(x)F(x)
\end{equation}
\

has a only a trivial solution $f(x)=0\in L^2([0,1],\mathbb{C})$. 
\end{Theorem}

In Theorem 4.6 $g(\pi)$ is a genus of compact surface which one can construct, using IET $(\lambda, \pi)$, and the property of interest to us is that $n\geq 2g(\pi)$ for $n\geq 2$.
\begin{Theorem}
Let $n\geq 3$ and $a_1,...,a_n\in\mathbb{C}:|a_k|=1,1\leq k\leq n$. Let function $\phi:[0, 1]\rightarrow\mathbb{C}$ be defined by
$\phi(x)=a_k$ if $x\in I_k$, for $1\leq k\leq n$.Then for almost all IETs $(\lambda, \pi)\in{\Delta}_{n-1}\times S_n^0$, and almost all $a_1,...,a_n$ and under condition $|F(x)|=1$, the equation

\begin{equation}
F(Tx)=\alpha \phi(x)F(x)
\end{equation}

has only trivial solutions $\alpha=1$, and $F(x)=$ constant 
\end{Theorem}

\begin{proof}
If $\phi:[0, 1]\rightarrow\mathbb{C}$ is defined by $\phi(x)=a_k=e^{2\pi i h_k}$, $h_k\in\mathbb{R}$, then the function $\alpha\phi(x)$ is defined by $\alpha\phi(x)=e^{2\pi i(h_k+t)}$, for some number $t\in\mathbb{R}$, such that $\alpha=e^{2\pi i t}$.
\

Let us define the set $\overline{W}=\{W+\mathbb{R}(1,...,1)\}=\{x\in {\mathbb{R}}^n| x=h+t(1,...,1)$, for some $h\in W$ and $t\in\mathbb{R}\}$.
As the Hausdorff dimension of $W$ is bounded by $g(\pi)$ then the Hausdorff dimension of $\overline{W}$ is bounded by $g(\pi)+1$ and so less than $n$. That implies that the Lebesgue measure of $\overline{W}$ is zero and the proof is complete.

\end{proof}

\begin{Theorem}

Let $\Theta: G\rightarrow U(1)$ be a non-trivial representation of $G$. Then for almost all triples $(\lambda,\pi,g)\in{\Delta}_{n-1}\times S^0_n\times G^n$ the following is true. For \textbf{all} $\alpha\in\mathbb{C}, |\alpha|=1$ the equation:

\begin{equation}
f(Tx)=\alpha\Theta(\phi(x))f   (x)
\end{equation}

\

has only a trivial solution $f(x)=0\in L^2([0, 1], \mathbb{C})$

\end{Theorem}

\begin{proof}
Given a triple $(\lambda,\pi,g)\in{\Delta}_{n-1}\times S^0_n\times G^n$ define a function $\Xi:[0,1]\rightarrow U(1)$ as $\Xi(x)=\Theta(\phi(x))$. By Theorem 4.7 there is a full measure set $P\in U(1)\times...\times U(1)$ such that for any $\alpha$ the equation
\begin{equation}
f(Tx)=\alpha\Xi(x)f(x)
\end{equation}

has only a trivial solution $f(x)=0$. 
\

The projection map $\rho=\Theta\times\dots\times\Theta: G^n\rightarrow {[U(1)]}^n$ is a locally trivial fiber bundle, and so ${\rho}^{-1}(P)$ has a full measure. The proof is complete.
\end{proof}

Given an IET $T$ of $n$ intervals and $g\in G^n$ we construct a simple function $\phi$ and a simple $G$-extension $T_{\phi}$.

\begin{Theorem} Let $n\geq 3$. For almost all triples $(\lambda, \pi, g)\in {\Delta}_{n-1}\times S^0_n\times G^n$ the $G$-extension $T_{\phi}:[0, 1]\times G\rightarrow [0, 1]\times G$ is weakly mixing.

\end{Theorem}

\begin{proof} By the result of Avila and Forni$[1]$ almost all $T=(\lambda, \pi)$ are weakly mixing for $n\geq 3$.  The Keynes-Newton criterion of weak mixing 4.2 for $T_{\phi}$ for a typical $\phi$ then immediately follows from Theorems 4.5 and 4.8.   
\end{proof}

\subsection{General U(1)-extensions and the case of two intervals.}

As Avila and Forni $[1]$ do not treat the case of two intervals, we have to make this case separately. In order to do it we first prove a general theorem 4.10 of independent interest. 
\

\

Let $T:X\rightarrow X$ be a measure preserving ergodic automorphism 
of a probability space $(X, \mu)$ and let $X=X_1\sqcup\dots\sqcup X_n$ be a finite partition of $X$ on measurable sets and ${\mathbb{T}}^n=\{z=(z_1,\dots, z_n)\in{\mathbb{C}}^n | |z_k|=1, 1\leq k\leq n\}$ . Using the partition of $X$ and $z$ we construct a
\textit{simple function} ${\phi}_z: X\rightarrow U(1)$ by formula ${\phi}_z(x)=z_k$, if $x\in X_k$.

\begin{defn}

 $z\in {\mathbb{T}}^n$ is a \textbf{generalized eigenvalue} if there exists  nonzero $f(x)\in L^2(X, \mathbb{C})$ such that:

\begin{equation}
f(Tx)={\phi}_{z}(x)f(x)
\end{equation}
and such an $f$ is called a \textbf{generalized eigenfunction}. 
\end{defn}

\begin{Lemma} Generalized eigenvalues form a multiplicative subgroup of ${\mathbb{T}}^n$.
\end{Lemma}

\begin{proof} If $f_z(x)$ is a generalized eigenfunction for generalized eigenvalue $z$ and $f_w(x)$ is a generalized eigenfunction for generalized eigenvalue $w$, then $f_z(x)f_w(x)$ is a generalized eigenfunction for $zw$ and $\overline{f_z(x)}$ is a generalized eigenfunction for
$\overline{z}$. Here if $z=(z_1,\dots, z_n)$ and $w=(w_1,\dots, w_n)$  then $zw=(z_1w_1,\dots, z_nw_n)$ and $\overline{z}=(\overline{z_1},\dots, \overline{z_n})$ 
\end{proof}

\begin{Theorem}Let $T:X\rightarrow X$ be a measure-preserving transformation of the probability space $(X,\mu)$. Let also $X=X_1\sqcup\dots\sqcup X_n$ be a finite partition of $X$ on measurable sets. For an element $z=(z_1,\dots, z_n)\in {\mathbb{T}}^n$ define 
${\phi}_z:X\rightarrow\mathbb{T}$ by formula ${\phi}_z(x)=z_k$ for $x\in X_k$, $1\leq k\leq n$. Then for almost all $z\in {\mathbb{T}}^n$ with respect to the Lebesgue measure on ${\mathbb{T}}^n$ the cohomological equation:

\begin{equation}
f(Tx)={\phi}_z(x)f(x)
\end{equation}

has only a trivial solution $f(x)=0\in L^2(X, \mathbb{C})$
\end{Theorem}

\begin{proof}

For given ${z\in\mathbb{T}}^n$ the existence of nonzero solution $f(x)\in L^2(X, \mathbb{C})$ of the cohomological equation:

\begin{equation}
f(Tx)={\phi}_z(x)f(x)
\end{equation}

means that $z$ is a generalized eigenvalue. By Lemma 4.4 the set $K$ of generalized eigenvalues is a subgroup of ${\mathbb{T}}^n$. By Lemma 6.3 $K$ is Borel. It is a classical fact that any Borel subgroup of ${\mathbb{T}}^n$ either has zero measure or coincides with ${\mathbb{T}}^n$.  If $K$ has zero measure, the proof is over. If $K={\mathbb{T}}^n$ this in particular implies that the \textit{diagonal subgroup} $\Delta=\{ z=(z_1,\dots, z_n)\in{\mathbb{T}}^n | z_1=\dots=z_n  \}$ is a subgroup of $K$. But for any $z=(\alpha,\dots, \alpha)\in\Delta$ the generalized eigenfunction ${\phi}_z(x)=\alpha$ and the cohomological equation (18) becomes

\begin{equation}
f(Tx)=\alpha f(x)
\end{equation}

which in turn implies that any $\alpha\in U(1)$ lies in the discrete spectrum of $T$, which is impossible as  the discrete spectrum of $T$ is at most countable.

\end{proof}

As an interesting immediate corollary of Theorem 4.10 ( and Keynes-Newton criterion of ergodicity) we have the following theorem of independent interest:

\begin{Theorem} Let $T$ be any ergodic IET of $n\geq 2$ intervals. Then for typical simple function $\phi: [0, 1]\rightarrow U(1)$ the $G$-extension $T_{\phi}:[0, 1]\times U(1)\rightarrow [0, 1]\times U(1)$ is ergodic.

\end{Theorem}

Here as usually the word \textit{typical} means Haar measure on $U(1)$. In Theorem 4.10 the \textit{ergodicity} of $T$ is the precise requirement, which clearly can not be weakened. It is an interesting question if ergodicity of an IET $T$ is enough to guarantee the ergodicity of a typical $G$-extension over $T$ in case of a general compact connected Lee group $G$. Let now $T$ be an interval exchange of $2$ intervals, characterized by parameter $\lambda\in [0, 1)$.

\begin{Theorem}

Let $\gamma: G\rightarrow U(1)$ be a nontrivial representation of $G$. Then for almost all triples $(\lambda, g_1, g_2)\in [0, 1)\times G^2$ the equation:

\begin{equation}
f(Tx)=\gamma(\phi(x))f(x)
\end{equation}

\

has only a trivial solution $f(x)=0\in L^2([0, 1], \mathbb{C})$

\end{Theorem}

\begin{proof}
Given a triple $(\lambda, g_1, g_2)\in [0, 1)\times G^2$  with irrational $\lambda$ define a function $\Xi:[0,1]\rightarrow U(1)$ as $\Xi(x)=\gamma(\phi(x))$. By Theorem 4.10 there is a full measure set $P\in U(1)\times...\times U(1)$ such that for any $\alpha$ the equation
\begin{equation}
f(Tx)=\Xi(x)f(x)
\end{equation}

has only a trivial solution $f(x)=0\in L^2([0, 1], \mathbb{C})$. As the projection map $\rho: G^n\rightarrow {[U(1)]}^n$ is a locally trivial fiber bundle, so ${\rho}^{-1}(P)$ has a full measure. The proof is complete.
\end{proof}
\

For an interval exchange $T$ of 2 intervals, characterized by parameter $\lambda\in [0, 1)$ and a pair $g=(g_1, g_2)\in G^2)$ we construct a simple function $\phi$ and $T_{\phi}$.
\begin{Theorem} For almost all pairs $(\lambda, g)\in [0, 1)\times G^2$ the $G$-extension $T_{\phi}$ is ergodic.

\end{Theorem}

\begin{proof} $T$ is ergodic for typical $\lambda$.  The Keynes-Newton criterion of ergodicity 4.1 for $T_{\phi}$ for a typical $\phi$ then immediately follows from Theorems 4.5 and 4.12.   
\end{proof}

\textbf{Remark}. In the assumptions of Theorem 4.12 the irrationality of $\alpha$ ( i.e ergodicity of $T$) is a necessary and sufficient condition. It is Theorem 4.5 dealing with higher-dimensional representations of $G$ does not allow to use only irrationality of 
$\alpha$ and instead makes us use a weaker assumption that $\alpha$ is typical. We strongly believe that the conclusion of Theorem 4.13 in fact holds for all irrational $\alpha$.

\subsection{Convergence to Haar measure for IET-generated random walk on G}
Given an ergodic transformation $T:X\rightarrow X$ of a probability space $X$, arbitrary partition of $X$ on $n$ measurable sets $X=X_1\sqcup X_2\sqcup\dots\sqcup X_n$ and a group $G$ one can define the $T$-generated random walk on $G$.  Namely for $x\in X$ and $g=(g_1, g_2, \dots, g_n)\in G^n$  one first creates an infinite word $w_x=w^1w^2\dots w^k\dots$ in the alphabet $\{1, 2,\dots, n\}$ by coding the trajectory of $x$ with respect to the partition 
$X=X_1\sqcup X_2\sqcup\dots\sqcup X_n$.  Then one creates an infinite sequence $\{a_x^k\}$ of elements of $G$, using $w^x$ and $g$ . Namely $a_x^k=g_{w^k}\cdot g_{w^{k-1}}\cdots g_{w^1}$. One also has a sequence of finitely supported measures 
${\mu}_x^k$ on $G$ by averaging Dirac measures along $\{a_x^k\}$, more precisely: ${\mu}_x^k=\frac{1}{k}({\delta}_{a_x^1}+{\delta}_{a_x^2}+\dots +{\delta}_{a_x^k})$.

An ergodic IET $T$ gives a natural partition of $[0, 1]$ on intervals $I_1,\dots, I_n$ and so construction above applies.  As an immediate consequence of ergodicity of $T_{\phi}$ for typical simple $G$-extension we have the following theorem.

\begin{Theorem} Let $G$ be a compact connected Lie group. Let $n\geq 2$ and $g=(g_1, g_2, \dots, g_n)\in G^n$. For almost all triples $(\lambda, \pi, g)\in {\Delta}_{n-1}\times S_n^0\times G^n$ and for almost all $x\in [0, 1]$ an IET-generated sequence ${\mu}_x^k$ converges to the normalized Haar measure $\mu$ on $G$.
\end{Theorem}

\section{Typical simple G-extensions are not measurably cohomologous.}
\subsection{ \textit{Cohomological equivalence of extensions}}

\begin{defn} Let $\phi:X\rightarrow G$ and $\psi: X\rightarrow G$ be two measurable functions from probability space $(X,\mu)$ to the compact Lie group $G$. Let also $T:X\rightarrow X$ be an ergodic transformation. The extensions $T_{\phi}$ and $T_{\psi}$ are said to be \textbf{measurably cohomologous} if there exists a measurable function $f:X\rightarrow G$ such that the equation 

\begin{equation}
f(Tx)\cdot\phi(x)=\psi(x)\cdot f(x)
\end{equation}

holds almost everywhere on $X$.  
\end{defn}

Existence of a measurable cohomology implies that $T_{\phi}$ and $T_{\psi}$ are isomorphic as measure-preserving maps. More precisely the  map 
$F: X\times G\rightarrow X\times G$, given by equation $F(x,g)=(x, f(x)g)$ satisfies $F\circ T_{\phi}= T_{\psi}\circ F$. 

\begin{Lemma} Let $G$ be a compact connected Lie group with Haar measure $\nu$ and $\overline{S}=\{(g, h) | h=aga^{-1}$ for some $a\in G\}$. Then $\overline{S}$ is a compact subset of $G\times G$ of measure zero.

\end{Lemma}

\begin{proof} The compactness of $\overline{S}$ obviously follows from compactness of $G$. It is enough to prove that $\nu(\overline{S})=0$. Let us fix any $g\in G$. If we prove that $\overline{S}_g=\{aga^{-1}, a\in G\}$ as a subset of $G$ has measure zero, then Fubini theorem applied to $G\times G$ would complete the proof.
\

\

For a fixed element $g\in G$ consider a smooth map $f:G\rightarrow G$ given by $f(u)=ugu^{-1}$. As $G$ is compact, there is an element 
$X\in\mathfrak{g}$, the Lie algebra of $G$, such that $g=exp(X)$. Clearly if $g\neq 1$ then $X\neq 0$ and $exp(X)exp(tX)=exp(tX)exp(X)$ for any $t\in\mathcal{R}$. Now pick any $u\in G$ and consider the sequence of elements $u_n=u\cdot exp(X/n)$. Then $u_n\rightarrow u$ and $f(u_n)=f(u)$. This implies that any $u\in G$ is a critical value of $f$. By Sard's Theorem $\nu(f(G))=\nu(\overline{S}_g)=0$.  

\end{proof}
Let $g=(g_1, \dots, g_n)\in G^n$ , $h=(h_1,\dots, h_n)\in G^n$, and let also 
\

$S=\{(g, h)\in G^n\times G^n | (h_1, g_1)\in\overline{S}\}$. Trivial consequence of Lemma 5.1 is that $S\in G^n\times G^n$ is a compact set of measure zero.

\begin{Lemma} For any biinvariant metric \textbf{d} on $G$ let 
$\{(g^m, h^m)\}\in G^n\times G^n$, $\{a^m\}\in G$ be sequences of elements such that $d(a^mg^m_1{(a^m)}^{-1}, h^m_1)\rightarrow 0.$ Then $d((g^m, h^m), S)\rightarrow 0$. 
\end{Lemma}

\begin{proof} The proof is identical to the proof of Lemma 4.2.
\end{proof}

\begin{Theorem} For a full measure set $\Omega$ of IETs, $\Omega\subseteq{\mathbb{R}}^n_+\times S^0_n$, $n\geq 2$ and for all $(g,h)\in G^n\times G^n$ the following property takes place.
\
Let $\phi(x):[0, |\lambda|]\rightarrow G$ and $\psi(x):[0, |\lambda|]\rightarrow G$ be simple functions, constructed by $(g,h)$. Assume that for $T\in\Omega$ the equation 

\begin{equation} 
f(Tx)\cdot\phi(x)=\psi(x)\cdot f(x)
\end{equation}
\

has a measurable solution $f:[0, 1]\rightarrow G$. Denote $({\lambda}^m, {\pi}^m, g^m)= {\mathcal{R}}^m(\lambda, \pi, g)$ and $({\lambda}^m, {\pi}^m, h^m)= {\mathcal{R}}^m(\lambda, \pi, h)$.  Then there exists an infinite set $E(T)\subseteq{\mathbb{Z}}_+$ and a sequence $a^m\in G$, such that $d(a^m\cdot g^m_1\cdot {(a^m)}^{-1},h^m_1)\rightarrow 0$, for $m\in E(T)$.
  
\end{Theorem}
\

\begin{proof}

Let us fix $\delta>0$ and let $\Omega$ be a full measure set of ergodic IETs satisfying the conclusion of Theorem 4.3. Pick $T\in\Omega$ and a pair $T_{\phi}$ and $T_{\psi}$ and assume that there exists a measurable function $f:X\rightarrow G$ satisfying equation $(23)$. Let $E\in{\mathbb{Z}}_+$  be an infinite set such that for each $m\in E$ , $T$ satisfies  $P_1(\epsilon, m)$ and $P_2(\epsilon, m)$.  If $m\in E$ and $I^m=[0, |{\lambda}^m|]$ then $P_1(\epsilon,m)$ implies  that $T^kI^m$ is an interval for $0\leq k<b$ ( $b$ depends on $m$) and also that $|{\cup}_{k=0}^{b-1}T^kI^m|\geq\epsilon|\lambda|$. As $m\rightarrow\infty$, $|I^m|\rightarrow 0$; therefore, if $m\in E$ is sufficiently large, then, by Lemma 6.1 there exist $k$ and $w\in G$ such that $0\leq k<b$ and

\begin{equation}
{\int}_{T^kI^m}d(f(x), w) dx<\delta|T^kI^m|
\end{equation}

Iterating equation (23) along the piece of trajectory $x, Tx, \dots, T^{n-1}x$ we obtain an equation 

\begin{equation}
f(T^nx)\cdot\phi(T^{n-1}x)\cdots\phi(x)=\psi(T^{n-1}x)\cdots\psi(x)\cdot f(x)
\end{equation}

\

Since $k<b$ the products $\phi(T^{k-1}x)\cdots\phi(x)$ and $\psi(T^{k-1}x)\cdots\psi(x)$ are independent on $x\in I^m$.  As $d$ is a bi-invariant metric it follows that there exists $w^m\in G$ such that

\begin{equation}
{\int}_{I^m} d(f(x), w^m)dx<\delta|I^m|
\end{equation}
\

Let $l_k^m$ denote the first return time of $I_k^m$ into $I^m$. We have the equation:

\begin{equation}
 f(T^{l_k^m}x)\cdot\phi(T^{l_k^m-1}x)\cdots\phi(x)=\psi(T^{l_k^m-1}x)\cdots\psi(x)\cdot f(x)
\end{equation}

which can be rewritten as :

\begin{equation}
 f(T^{l^k_m}x)\cdot g_m^k=h_m^k\cdot f(x)
\end{equation}
\

and correspondingly:

\begin{equation}
 (f(T^{l_k^m}x)\cdot {f(x)}^{-1})\cdot (f(x) \cdot g^m_k \cdot {f(x)}^{-1})=h^m_k
\end{equation}
\

From equation 24 it follows that the set $\{x\in I^m: d(f(x), w^m)\geq\sqrt{\delta}\}$ has measure at most $\sqrt{\delta}|I^m|$. The property $P_2(\epsilon, m)$ guarantees that if $\delta$ is small enough then there is an $x\in I^m$ such that 
$d(f(T^{l_k^m}x), w^m)<\sqrt{\delta}$ and $d(f(x), w^m)<\sqrt{\delta}$ which implies $d(f(T^{l_k^m}x), f(x))<2\sqrt{\delta}$ 

\

As $d$ is a bi-invariant metric on $G$ and $d(f(T^{l_k^m}x), f(x))\leq 2\sqrt{\delta}$ for chosen $x$ then by choosing $k=1$ and $a_m=f(x)$ we complete the proof.

\end{proof}
\

Let $n\geq 2$. As usual having an irreducible IET $T$ and elements $g, h\in G^n$ we construct simple functions $\phi$, $\psi$.

\begin{Theorem} For almost all triples $(T, \phi, \psi)$ the $G$-extensions $T_{\phi}$ and $T_{\psi}$ are not measurably cohomologous.

\end{Theorem}

\begin{proof} 

Assume that for some quadruple $(\lambda,\pi,g)\in{\Delta}_{n-1}\times S^0_n\times G^n\times G^n$ there exists a nonzero solution $f(x)$ to the equation (22). Then by Theorem 5.1 there exists an infinite set $E\subseteq{\mathbb{Z}}_+$ and a sequence $a^m\in G$, such that $d(a^m\cdot g^m_1\cdot {(a^m)}^{-1},h^m_1)\rightarrow 0$, for $m\in E$.  Then Lemma 5.2 implies that $d((g^m, h^m), S)\rightarrow 0$.
 \
 
 It is enough then to prove that for any sequence ${\Gamma}^m:G^n\times G^n\rightarrow G^n\times G^n$ consisting of direct squares of Rauzy maps $A\times A$ and $B\times B$, the set
 $D=\{(g, h)\in G^n\times G^n | d({\Gamma}^m(g), S)\rightarrow 0\}$ has measure zero. 
 \
 
 \
 
 Choose a positive integer $p$ and consider a set 
 $S_p=\{(g, h)\in G^n\times G^n | d((g, h), S)<1/p\}$. Then clearly the set $D$ is eventually in $S_p$ under the sequence ${\Gamma}^m$. So by Lemma 4.3 for any $p$, $\nu(D)\leq\nu(S_p)$. As the set $S$ is compact it implies that $S=\bigcap S_p$. As $S_p$ is a monotone sequence of sets, $\nu(S_p)\rightarrow\nu(S)=0$ and so $\nu(D)=0$.

\end{proof}

\subsection{Open problems}

\textbf{Problem 1.} Does the analog of Theorem 4.11 hold for any ergodic IET $T$ and any compact connected Lie group $G$? Namely if $T$ is an ergodic IET, is it true that for typical simple function $\phi$ the $G$-extension $T_{\phi}$ is ergodic?
\

\

\textbf{Problem 2.} Does the analog of Theorem 4.9 hold for any weakly-mixing IET $T$? Namely if $T$ is a weakly-mixing IET, is it true that for typical simple function $\phi$ the $G$-extension $T_{\phi}$ is weakly-mixing?
\

\

\textbf{Problem 3.} If the Lie group $G$ is noncompact, then there is no finite Haar measure. However there could be finite covolume lattices in unimodular Lie group $G$ and the question of typical weak-mixing/ ergodicity makes sense. For example is it true that for typical IET $T$ and typical simple function $\phi:[0, 1]\rightarrow SL(d,R)$ the skew-product $T_{\phi}:[0, 1]\times SL(d, \mathbb{R})/SL(d,\mathbb{Z})\rightarrow [0, 1]\times SL(d, \mathbb{R})/SL(d, \mathbb{Z})$ is ergodic? The question is intriguing as the Extended Rauzy-Veech Renormalization $\mathcal{R}$ is well-defined for noncompact Lie groups, however there is no known analog of Keynes-Newton criterion due to the presence of infinite-dimensional irreducible representations.
\

\

\textbf{Problem 4.} Does the analog of Theorem 5.2 hold for any Lie group $G$, not necessarily compact? The argument used in this paper can not be directly generalized due to the absence in general of a biinvariant metric \textbf{d} on $G$, however the left and right-invariant  metrics always exist and it might happen that some delicate refinement of the present argument would still yield the result.
\

\

\textbf{Problem 5.}  Does the analog of Theorem 5.2 actually hold for any ergodic transformation $T$? Namely let $G$ be a compact connected Lie group and $T:X\rightarrow X$ be an ergodic automorphism of the probability space $(X, \mu)$. Let also 

$X=X_1\sqcup\dots\sqcup X_n$ be a finite partition of $X$ onto measurable sets and $\phi: X\rightarrow G$ and $\psi:X\rightarrow G$ be \textit{simple functions} with respect to the partition. Is it true that  for typical pair $\phi$ and $\psi$ with respect to the Haar measure on $G$ the $G$-extensions $T_{\phi}$ and $T_{\psi}$ are not measurably cohomologous? If $G=U(1)$ this statement immediately follows from Theorem 4.10.

\newpage
\section{Appendix.}
\subsection{Simple approximation.}
\begin{defn} Let $M$ be a compact smooth manifold and $[0, 1]=I_1\sqcup\dots\sqcup I_n$ be a finite partition of the interval $[0, 1]$ on subintervals. A function $f:[0, 1]\rightarrow M$ is called \textbf{simple} if it is constant on each of the intervals 
$I_k$, $1\leq k\leq n$.

\end{defn}

\begin{Lemma} Let $M$ be a compact manifold with metric $d$ induced by some Riemannian metric $g$ and let $\phi: [0, 1]\rightarrow M$ be a measurable function and $0<\mu\leq 1$ be a fixed number. Then for any $\gamma>0$ there exists $\delta>0$ such that for any partition $\tau$ of the interval $[0, 1]$ on subintervals with $diam(\tau)<\delta$ and for any set $X\subseteq [0, 1]$, $\mu(X)\geq\mu$ consisting of elements of $\tau$ there exists an interval $J\in\tau$, $J\subseteq X$, and $c\in M$ such that ${\int}_{J}d(\phi(x), c)dx<\gamma|J|$
\end{Lemma}

Here $\mu(X)$ means the Lebesgue measure of the set $X$.

\begin{proof} By isometric embedding of $M$ into appropriate ${\mathbb{R}}^n$ and applying Lusin's theorem we may assume that for small enough $\delta>0$ there is a simple function $f:[0, 1]\rightarrow M$ such that 
${\int}_0^1d(\phi(x), f(x))dx<\gamma\mu$. Assume that for each interval $J\in\tau$, $J\subseteq X$ we have that ${\int}_{J}d(\phi(x), f(x))dx\geq\gamma|J|$. But then summing up for all such $J$ we have that
${\int}_{X}d(\phi(x), f(x))dx\geq\gamma\mu(X)=\gamma\mu$ which gives a contradiction.
\

That implies that there is an interval $J\in\tau$, $J\subseteq X$, such that ${\int}_{J}d(\phi(x), f(x))dx<\gamma|J|$. As $f(x)$ is a simple function, $f(x)=c$ on $J$ and so the conclusion follows.
\end{proof}

\subsection{Borel measurability of the generalized spectrum.}
\begin{defn}$\{$Generalized eigenvalue$\}$. Let $H_1, H_2,\dots H_n $ be separable Hilbert spaces and let 
  $T: H_1 \oplus H_2 \oplus\cdots\oplus H_n\rightarrow H_1 \oplus H_2 \oplus\cdots\oplus H_n $ be a bounded linear operator. Then a vector 
  $\lambda=(\lambda_1,\lambda_2,\dots , \lambda_n)\in {\mathbb{C}}^n$ is called a \textbf{generalized eigenvalue} if there exists 
  $x=(x_1, x_2,\dots , x_n) \in H_1 \oplus H_2 \oplus\cdots\oplus H_n$, $x\neq 0$ such that $Tx=({\lambda}_1 x_1, {\lambda}_2 x_2,\dots, {\lambda}_n x_n)$
\end{defn}

\begin{defn}$\{$Generalized point spectrum$\}$.  The \textbf{generalized point spectrum} $Sp(T)\subseteq  {\mathbb{C}}^n$ of $T$ is a set of all of its generalized eigenvalues.
\end{defn}

We now make a general characterization of the nontriviality of $Ker(T)$, where $T\in B(H)$ is a bounded operator on the separable Hilbert Space $H$.

\begin{Lemma} Assume that $P\subseteq U(H)$ is any dense countable subset of the unit sphere $U(H)=\{x\in H | ||x||=1\}\subseteq H$ and $T\in B(H)$. Then $Ker(T)\neq 0$ if and only if the following condition takes place. There exists an element $a\in P$, $a\neq 0$ and such that for any $k\in{\mathbb{Z}}_+$ there exists $b_k\in P$ such that:
\

\

$1) ||a-b||\leq \frac{1}{2}$ 
\

\

$2) T(b_k)\leq\frac{1}{k}$

\end{Lemma}

\begin{proof} As \textbf{if} direction of the Lemma is immediate, we only prove the \textbf{only if} direction. Consider the sequence $\{b_k\}$ satisfying (1) and (2). As $U(H)$ is \textit{weakly compact} then by taking a subsequence we may assume that 
for some $b\in H$, $b_k\rightarrow b$ \textit{weakly}. Now $b\neq 0$ because if $\Psi: H\rightarrow\mathbb{C}$ is a linear functional given by formula $\Psi(x)=\langle x, a\rangle$, then condition (1) implies that 
$\Psi(b)=\lim\Psi(b_k)\geq\frac{1}{2}$.
\

\

As any $T\in B(H)$ is a continuous linear map in the \textit{weak} topology on $H$, $T(b_k)\rightarrow T(b)$ \textit{weakly} and as $T(b_k)\rightarrow 0$ \textit{strongly} by condition (2) it follows that $T(b)=0$ as the strong convergence implies weak convergence on $H$.

\end{proof}

\begin{Lemma}

For any bounded operator $T$, $Sp(T)$ is a Borel subset of ${\mathbb{C}}^n$.

\end{Lemma}

\begin{proof} For $\lambda\in{\mathbb{C}}^n$ let $T_{\lambda}:H_1 \oplus H_2 \oplus\cdots\oplus H_n\rightarrow H_1 \oplus H_2 \oplus\cdots\oplus H_n $ be defined by formula
$T_{\lambda}(x_1,\dots, x_n)=T(x)-({\lambda}_1 x_1,\dots, {\lambda}_n x_n)$. In these notations $Sp(T)=\{\lambda\in{\mathbb{C}}^n | Ker(T_{\lambda})\neq 0\}$.
\

\

Let us now choose some dense sequence $\{b_k\}\subseteq U(H)$. 
For a triple $(p, k, q)\in {\mathbb{Z}}_+^3$ consider the set $S_{pkq}=\{\lambda\in{\mathbb{C}}^n: ||b_p-b_q||<\frac{1}{2}$ and $||T_{\lambda}(b_q)||<\frac{1}{k}\}$.
 By Lemma 6.2 we can write:
 \

\

 $Sp(T)=\{ \lambda\in {\mathbb{C}}^n |\exists p\forall k\exists q$ such that: $||b_p-b_q||<\frac{1}{2}$ and $||T_{\lambda}(b_q)||<\frac{1}{k}\}$,
 \
 
 \
 
  which can be rewritten as $Sp(T)=\bigcup\limits_{p}\bigcap\limits_{k}\bigcup\limits_{q} S_{pkq}$. As $S_{pkq}\subseteq{\mathbb{C}}^n$ is clearly open, the conclusion of Lemma follows.
\end{proof}

\end{document}